\documentclass{amsart}

\usepackage{amssymb}
\usepackage{amsthm}
\usepackage{amsmath}
\usepackage{hyperref}

\usepackage[matrix, arrow]{xy}
\xyoption{arrow}
\xyoption{curve}
\theoremstyle{plain}\newtheorem{Theorem}{Theorem}[section]
\theoremstyle{plain}
\theoremstyle{plain}
\theoremstyle{plain}\newtheorem{Lemma}[Theorem]{Lemma}
\theoremstyle{plain}\newtheorem{Proposition}[Theorem]{Proposition}
\theoremstyle{plain}
\theoremstyle{definition}
\theoremstyle{definition}
\theoremstyle{definition}
\theoremstyle{definition}
\theoremstyle{definition}
\theoremstyle{definition}\newtheorem{Remark}[Theorem]{Remark}
\theoremstyle{definition}
\theoremstyle{definition}

\def\Aut{\mathrm{Aut}}                    
\def\Br{\mathrm{Br}}

\def\chr{\mathrm{char}}

\def\dim{\mathrm{dim}}

\def\End{\mathrm{End}}

\def\ker{\mathrm{ker}}

\def\Im{\mathrm{Im}}

\def\Irr{\mathrm{Irr}}

\def\uZ{\underline{Z}}

\def\op{\mathrm{op}}

  \def\pr{\mathrm{pr}}

\def\soc{\mathrm{soc}}

\def\Tr{\mathrm{Tr}}

\def\tn{\textnormal}
\usepackage{tikz-cd}
\usepackage{pgf,tikz}
\usetikzlibrary{arrows}
\usetikzlibrary{cd}
\usepackage{graphicx}

\makeindex
\title{A $9$-dimensional algebra which is not a block 
of a finite group} 
\author{Markus Linckelmann and  William Murphy} 
\date{\today}

\begin{document}

\begin{abstract}
We rule out a certain $9$-dimensional algebra over an algebraically
closed field to be the basic algebra of a block of a finite group,
thereby completing the classification of basic algebras of 
dimension at most $12$ of blocks of finite group algebras.
\end{abstract}

\maketitle

\section{Introduction}

Basic algebras of block algebras of finite groups over an
algebraically closed field of dimension at most $12$ have been 
classified in \cite{LinICRA16}, except for one $9$-dimensional
symmetric algebra over an algebraically closed field $k$ of 
characteristic $3$ with two isomorphism classes of simple modules 
for which it is not known whether it actually arises as a basic
algebra of a block of a finite group algebra. 
The purpose of this paper is to show that this algebra does not
arise in this way.
It is shown in \cite[Section 2.9]{LinICRA16} that if 
$A$ is a $9$-dimensional basic algebra over an algebraically closed 
field $k$ of prime characteristic $p$ with two isomorphism classes of 
simple modules such that $A$ is isomorphic to a basic algebra of a 
block $B$ of $kG$ for some finite group $G$, then the algebra $A$ has 
the Cartan matrix 
$$
C = \left(\begin{array}{cc} 5 & 1 \\ 1 & 2 \end{array}\right), \ \
$$
Since the elementary divisors of $C$ are $9$ and $1$, it follows that
$p=3$ and that a defect group $P$ of $B$ is either cyclic (in which 
case $A$ is a Brauer tree algebra) or $P$ is elementary abelian of 
order $9$. We will show that the second case does not arise.

\begin{Theorem} \label{main1}
Let $k$ be an algebraically closed field of prime characteristic $p$.
Let $G$ be a finite group and $B$ a block of $kG$ with Cartan matrix
$C$ as above.
Then $p=3$, the defect groups of $B$ are cyclic of order $9$, and
$B$ is Morita equivalent to the Brauer tree algebra of the tree with
two edges, exceptional multiplicity $4$ and exceptional vertex at the
end of the tree.
\end{Theorem}

The proof of Theorem \ref{main1} proceeds in the following stages. We first 
identify  in Theorem \ref{main2} any hypothetical basic algebra $A$ of 
a block with Cartan matrix $C$ as above and a noncyclic defect group. 
It turns out that there is only one candidate algebra, up to 
isomorphism. In Section \ref{AStructure} we give a description of the 
structure of this candidate $A$, and we show in Theorem \ref{main5} that 
$A$ is not isomorphic to a basic algebra of a block.
The proof of Theorem \ref{main2} amounts essentially to
filling in the details in \cite[section 2.9]{LinICRA16}. For the proof
of Theorem \ref{main5} we combine a stable equivalence of Puig
\cite{Puabelien}, a result of Brou\'e in \cite{BroueEq} on the invariance of 
stable centres under stable equivalences of Morita type, results of Kiyota 
\cite{Kiyota} on blocks with an elementary abelian defect group of order 
$9$, and properties of blocks with symmetric stable centres
from \cite{KessarLinckelmann}. A slightly different approach to proving
Theorem \ref{main5} is outlined in the last section, first showing in
Proposition \ref{Further1} a more precise result on the stable equivalence
class of $A$, and then using Rouquier's stable equivalences for blocks 
with elementary abelian defect groups of rank $2$ from 
\cite{Rouqstable}.  

\section{The basic algebra $A$ of a noncyclic block with Cartan matrix $C$} \label{A}

The following result is stated in \cite{Eatonwiki} without proof; 
for the convenience of the reader we give a detailed proof, following 
in part the arguments in \cite[Section 2.9]{LinICRA16}. 

\begin{Theorem} \label{main2}
Let $k$ be an algebraically closed field of prime characteristic $p$.
Let $A$ be a basic algebra with Cartan matrix 
$$
C = \left(\begin{array}{cc} 5 & 1 \\ 1 & 2 \end{array}\right), \ \
$$
such that $A$ is Morita equivalent to a block $B$ of $kG$, for some
finite group $G$, with a noncyclic defect group $P$. Then $p=3$, we
have $P\cong$ $C_3\times C_3$, and $A$ is isomorphic to the algebra
given by the quiver

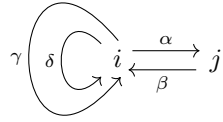
\vspace{-12mm}
$$
\begin{tikzcd}
i \arrow[ start anchor={[xshift=1.5ex]},end anchor={[xshift=1.5ex]}, in=230,out=130,loop,distance=25mm,swap, "\gamma"] \arrow[in=225,out=135,loop,distance=10mm,swap,"\delta"] \arrow[r, start anchor={[xshift=-0.15ex]}, end anchor={[xshift=0.15ex]}, shift left=1.5, "\alpha"] &  j  \arrow[l,  end anchor={[xshift=-0.5ex]}, start anchor={[xshift=-0.15ex]}, shift left=1.5, "\beta"]
\end{tikzcd}
$$
\vspace{-8mm}

\noindent with relations
$\delta^2=\gamma^3=\alpha\beta$, $\delta\gamma=$
$\gamma\delta=0$, $\delta\alpha=\gamma\alpha=0$, and
$\beta\delta=\beta\gamma=0$. In particular, we have
$$|\Irr(B)| = \dim_k(Z(A))=6\ ,$$
and the decomposition matrix of $B$ is equal to 
$$
D = \left(\begin{array}{cc} 1 & 0 \\ 1 & 0 \\ 1 & 0 \\ 1 & 0\\
1 & 1 \\ 0 & 1  \end{array}\right) \ \ .
$$
\end{Theorem}

\begin{proof}
As mentioned above, since the elementary divisors of the Cartan
matrix $C$ are $1$ and $9$, it follows that $p=3$ and that $B$ has a 
defect group $P$ of order $9$. Since $P$ is assumed to be noncyclic,
it follows that $P\cong$ $C_3\times C_3$.

Let $\{i,j\}$ be a primitive decomposition of $1$ in $A$. Set
$S=$ $Ai/J(A)i$ and $T=$ $Aj/J(A)j$. It follows from the entries of the
Cartan matrix that we may choose notation such that $Ai$ has composition
length $6$ and $Aj$ has composition length $3$. Since the top and
bottom composition factor of $Aj$ are both isomorphic to $T$, it follows
that $Aj$ is uniserial, with composition factors $T$, $S$, $T$ (from
top to bottom).

We label the two vertices of the quiver of $A$ by $i$ and $j$.
The quiver of $A$ contains a unique arrow from $i$ to $j$ and no
loop at $j$ because $J(A)j/J(A)^2j\cong$ $S$. Thus there is an
$A$-homomorphism
$$\alpha : Ai \to Aj$$
with image $\Im(\alpha)=V$ uniserial of length $2$, with composition
factors $S$, $T$. Since $Aj$ is uniserial of length $3$, it follows that 
$V=$ $J(A)j$ is the unique submodule of length $2$ in $Aj$.

The symmetry of $A$ implies that the quiver of $A$ contains a path from 
$j$ to $i$. This forces that the quiver of $A$ has an arrow from $j$ to 
$i$. Since the Cartan matrix of $A$ implies that $Ai$ has exactly one 
composition factor $T$, it follows that the quiver of $A$ contains 
exactly one arrow from $j$ to $i$. This arrow corresponds to an 
$A$-homomorphism
$$\beta : Aj\to Ai$$
which is not injective as $Aj$ is an injective module. Thus 
$U=$ $\Im(\beta)$ is a submodule of $Ai$  of length at most $2$. 
The length of $U$ cannot be $1$, because the top composition factor
of $U$ is $T$, but the unique simple submodule of $Ai$ is isomorphic
to $S$. Thus $U$ is a uniserial submodule of length $2$ of $Ai$,
with composition factors $T$, $S$. It follows that $\beta\circ\alpha$
is an endomorphism of $Ai$ with image $\soc(Ai)$.

Since $U=$ $\Im(\beta)$ and $\beta$ corresponds to an arrow in the quiver
of $A$, it follows that $U$ is not contained in 
$J(A)^2i$. Thus the simple submodule $U/\soc(Ai)$ of $J(A)i/\soc(Ai)$ is
not contained in the radical of $J(A)i/\soc(Ai)$, and therefore must be
a direct summand. Let $M$ be a submodule of $Ai$ such that $M/\soc(Ai)$
is a complement of $U/\soc(Ai)$ in $J(A)i/\soc(Ai)$. Then
$$J(A)i = U + M$$
$$\soc(Ai) = U\cap M$$
and, by the Cartan matrix, $M$ has composition length $4$, and
all composition factors of $M$ are isomorphic to $S$, and $\soc(M)=$
$\soc(Ai)$. 
Equivalently, $M/\soc(Ai)$ has length $3$, with all composition factors
isomorphic to $S$. We rule out some cases.

\smallskip
\noindent {\bf (1)}
$M/\soc(Ai)$ cannot be semisimple. Indeed, if it were semisimple, then
$J(A)i/\soc(Ai)=$ $U/\soc(Ai)\oplus M/\soc(Ai)$ would be semisimple.
This would imply that $J(A)^3i=\{0\}$. Since also $J(A)^3j=\{0\}$,
it would follow that $\ell\ell(A)=3$. But a result of Okuyama in 
\cite{Okcube} rules this out. Thus $M/\soc(Ai)$ is not semisimple. 

\smallskip
\noindent {\bf (2)}
$M/\soc(Ai)$ cannot be uniserial. Indeed, if it were, then the quiver of
$A$ would have a unique loop at $i$, corresponding to an endomorphism
$\gamma$ of $Ai$ mapping $Ai$ onto $M$ (with kernel necessarily equal to 
$U$ because $M$ has no composition factor isomorphic to $T$). 
Then $\gamma^5=0$ and $\gamma^4$ has image $\soc(Ai)\cong$ $S$. 

By construction, $\alpha$ maps $U$ to $\soc(Aj$ and $\beta$ maps $V$ to
$\soc(Ai)$. Thus $\beta\circ\alpha$ sends $Ai$ onto $\soc(Ai)$, and
$\alpha\circ\beta$ sends $Aj$ onto $\soc(Aj)$. Thus $\gamma^4$ and
$\beta\circ\alpha$ differ at most by nonzero  scalar. We may choose
$\alpha$ such that $\gamma^4=\beta\circ\alpha$. 

The homomorphism $\alpha$ sends $M$ to zero, because $Aj$ contains no
simple submodule isomorphic to $S$. Thus $\alpha\circ\gamma=0$. 
Also, since $U$ is the kernel of $\gamma$, we have $\gamma\circ\beta=0$.

It follows from these relations that $A$ is a Brauer tree algebra, of a
tree with two edges, exceptional multiplicity $4$, and exceptional vertex
at an end of the Brauer tree. This would force $P$ to be cyclic, 
contradicting the current assumption that $P\cong$ $C_3\times C_3$. 

\smallskip
\noindent {\bf (3)}
$M/\soc(Ai)$ cannot be indecomposable. Indeed, if it were, then it would
have Loewy length $2$ because it has composition length $3$, but is 
neither of length $1$ (because it is not semisimple) nor of length $3$ 
(because it is not uniserial). But then either its socle or its top is 
simple, and therefore it would have to be either a quotient of $Ai$, or a 
submodule of $Ai$. We rule out both cases. 

Suppose first that $M/\soc(Ai)$ is a quotient of $Ai$. Note that then
$M$ itself has a simple top, isomorphic to $S$, hence is a quotient of 
$Ai$ because $Ai$ is projective. Comparing composition lengths yields
$M\cong$ $Ai/U$. But also $U+M=J(A)i$, so the image of $M$ in 
$Ai/U$ is the unique maximal submodule $J(A)i/U$ of $Ai/U\cong $ $M$.
Thus $J(A)M$ is the unique maximal submodule of $M$, and that maximal
submodule is isomorphic to a quotient of $M$, hence has itself a unique
maximal submodule. This however would imply that $M/\soc(Ai)$ is 
uniserial of length $3$, which was ruled out earlier.  

Suppose finally that $M/\soc(Ai)$ is a submodule of $Ai$. Then it must be
a submodule of $M$, because it does not have a composition factor $T$. 
Moreover, $M$ and the image of $M/\soc(Ai)$ in $M$ both have the same
simple socle $\soc(Ai)$. 
Thus $M/\soc(Ai)$ divided by its socle (which is simple) is a submodule 
of $M/\soc(Ai)$, which has a simple socle. Thus the first and second 
socle series quotients are both simple, again forcing $M/\soc(Ai)$ to be
uniserial, which is not possible.  

\smallskip
\noindent {\bf (4)}
Combining the above, it follows that $M/\soc(Ai)$ is a direct sum of 
$S$ and a uniserial module of length $2$ with both composition factors
$S$. That is, we have
$$M = M_1 + M_2$$
for some submodules $M_i$ of $M$ with
$$M_1\cap M_2= \soc(Ai)=\soc(M)$$
$$M_1/\soc(Ai)\cong S$$
and $M_2/\soc(A_i)$ uniserial of length $2$. It follows that
$M_1$ and $M_2$ are uniserial, of lengths $2$ and $3$, respectively.

\smallskip
We choose now $M_2$ as follows. By construction, we have a direct sum
$$J(A)i/\soc(Ai) = U/\soc(Ai) \oplus M_1/\soc(Ai) \oplus M_2/\soc(Ai)$$
Thus we have
$$J(A)i/(U + M_1) \cong (J(A)i/\soc(Ai))/(U/\soc(Ai)\oplus 
M_1/\soc(Ai))\cong M_2/\soc(Ai)\ .$$
This is a uniserial module with two composition factors isomorphic to $S$.
Thus $Ai/(U+M_1)$ is uniserial with three composition factors isomorphic
to $S$, because $Ai/J(A)i\cong$ $S$. Since in particular its socle is
simple, isomorphic to $S$, this module is isomorphic to a submodule of 
$Ai$.
Choose an embedding $A/(U + M_1) \to Ai$ and replace $M_2$ by the image
of this embedding. Then the composition of canonical maps
$$\gamma : Ai \to Ai/(U + M_1) \to Ai$$
is an $A$-endomorphism of $Ai$ with kernel $U+M_1$ and uniserial image
$M_2$ of length three. Note that $M_1$ is uniserial of length two, so
both a quotient and a submodule of $Ai$. Thus there is an endomorphism
$$\delta : Ai \to Ai$$
with image $M_1$. Since $M_1\subseteq$ $\ker(\gamma)$, we have
$$\gamma\circ\delta = 0\ .$$
We show next that we also have 
$$\delta\circ\gamma = 0\ .$$
One way to see this is to observe that this is a calculation in the 
split local $5$-dimensional symmetric algebra $\End_A(Ai)\cong$ 
$(iAi)^\op$, which as a consequence of \cite[B. Theorem]{Kuesmall}, is 
commutative.

There is a (slightly more general) argument that works in this 
case. Since the $A$-module $Ai$, and hence also the image of $\gamma$,  
is generated by $i$, it suffices to show that $\delta(\gamma(i))=0$. 
Now since $\gamma\circ\delta=0$, we have 
$$0=\gamma(\delta(i))= \gamma(\delta(i)i)=\delta(i)\gamma(i)$$
Note that $\delta(i)=\delta(i^2)=i\delta(i)\in$ $iAi$, and similarly,
$\gamma(i)\in$ $iAi$. Since $\Im(\delta)=$ $M_2$ has length $2$,we 
have $\Im(\delta)\subseteq$ $\soc^2(A)$. Thus $\delta(i)\in$
$\soc^2(A)\cap iAi\subseteq$ $\soc^2(iAi)$, and since $iAi$ is symmetric, 
we have $\soc^2(iAi)\subseteq$ $Z(iAi)$. It follows that
$$\delta(i)\gamma(i)=\gamma(i)\delta(i)=\delta(\gamma(i)i)=
\delta(\gamma(i))$$
whence $\delta(\gamma(i))=0$, and so $\delta\circ\gamma=0$ by the 
previous remarks. Thus $M_2\subseteq$  $\ker(\delta)$. Since 
$\Im(\delta)=M_1$ has no composition factor $T$, it follows that 
$U\subseteq$ $\ker(\delta)$. Together we get that $U+M_2\subseteq$ 
$\ker(\delta)$. Comparing composition lengths yields
$$\ker(\delta)=U+M_2\ .$$
This implies that 
$$\ker(\delta)\cap \Im(\delta)=\soc(Ai)$$
$$\ker(\gamma)\cap \Im(\gamma)=\soc(Ai)$$
and hence the endomorphisms $\delta^2$ and $\gamma^3$ both map $Ai$
onto $\soc(Ai)$. Thus they differ by a nonzero scalar. Up to adjusting 
$\delta$, $\beta$, we may therefore assume that
$$\delta^2 = \gamma^3=\beta\circ\alpha$$
Since $\ker(\alpha)$ contains $M_1+M_2$, it follows that 
$$\alpha\circ\delta = \alpha\circ\gamma = 0\ .$$
By taking these relations into account, it follows that $\End_A(A)$
is spanned $k$-linearly by the set
$$\{i, j, \alpha, \beta, \gamma, \gamma^2, \delta, \delta^2,
\alpha\circ\beta\}$$
so this is a basis of $\End_A(A)$. We have identified here $i$, $j$ 
with the canonical projections of $A$ onto $Ai$ and $Aj$. 
Note that $\End_A(A)$ is the algebra opposite to $A$. This accounts
for the reverse order in the relations of the generators in $A$
(denoted abusively by the same letters). This shows that the quiver
with relations of $A$ is as stated. The equation $C=$ $(D^t)D$
implies that the second column of $D$ has exactly two nonzero
entries and that these are equal to $1$. The first row has either 
five entries equal to $1$, which yields $|\Irr(B)|=6$ and the
decomposition matrix $D$ as stated. Or the first row has one entry $2$ 
and one entry $1$. This would lead to a decomposition matrix of the
form 
$$
D = \left(\begin{array}{cc} 2 & 0 \\ 
1 & 1 \\ 0 & 1  \end{array}\right) \ \ .
$$
In particular, this would yield $|\Irr(B)|=3$. But this is not possible, 
since $\dim_k(Z(A))$ is clearly greater than $3$; indeed, $Z(A)$ 
contains the linearly independent elements $1$, $\delta$, $\gamma$, 
$\gamma^2$. This concludes the proof.
\end{proof}

\section{The structure of the algebra $A$}\label{AStructure}

Let $k$ be an algebraically closed field. Throughout this section we denote by 
$A$ the $k$-algebra given in Theorem \ref{main2}. 
We keep the notation of this theorem and identify the
generators $i$, $j$, $\alpha$, $\beta$, $\gamma$, $\delta$ with their
images in $A$.

\begin{Lemma} \label{LemmaA}
\mbox{}
\begin{enumerate}
\item[{\rm (i)}]
The set $\{i, j, \alpha, \beta, \beta\alpha, \gamma, \gamma^2, 
\delta, \delta^2\}$ is a $k$-basis of $A$.
 
\item[{\rm (ii)}]
The set $\{\alpha, \beta, \alpha\beta-\beta\alpha\}$ is a $k$-basis
of $[A, A]$.

\item[{\rm (iii)}]
The set $\{1, \gamma, \gamma^2, \delta, \delta^2, \beta\alpha\}$ 
is a $k$-basis of $Z(A)$. 

\item[{\rm (iv)}] The set $\{\alpha\beta, \beta\alpha \}$ is a $k$-basis of $\soc(A)$.

\end{enumerate}
\end{Lemma}

\begin{proof}
This follows immediately from the relations of the quiver of $A$.
\end{proof}

\begin{Lemma} \label{symmLemma}
There is a unique symmetrising form $s : A\to k$ such that 
$$s(\alpha\beta)=s(\beta\alpha)=1$$ 
and such that 
$$s(i)=s(j)=s(\alpha)=s(\beta)=s(\gamma)=s(\gamma^2)=s(\delta)=0$$
The dual basis with respect to the form $s$ of the basis 
$$\{i, j, \alpha, \beta, \beta\alpha, \gamma, \gamma^2, 
\delta, \delta^2\}$$
is, in this order, the basis
$$\{\alpha\beta, \beta\alpha, \beta, \alpha, j, \gamma^2, \gamma,
\delta, i\}$$
\end{Lemma}

\begin{proof}
Straightforward verification.
\end{proof}

See \cite[\S 5.B]{BroueEq} or \cite[Definition 2.16.10]{LinBlockI} for 
details regarding the definitions and some properties of the 
{\it projective ideal} $Z^\pr(A)$ in $Z(A)$ and the {\it stable centre} 
$\uZ(A)=$ $Z(A)/Z^\pr(A)$. 

\begin{Lemma} \label{projcenterLemma}
Let $\chr(k)=3$. The projective ideal $Z^\pr(A)$ is one-dimensional, 
with basis $\{\alpha\beta-\beta\alpha\}$, we have an isomorphism of
$k$-algebras
$$\underline{Z}(A) \cong k[x,y]/(x^3-y^2, xy, y^3)$$
induced by the map sending $x$ to $\gamma$ and $y$ to $\delta$, 
and after identifying $x$ and $y$ with their in the quotient, the 
following statements hold:

\begin{enumerate}
\item[{\rm (i)}] 
The set $\{1,x,x^2,y,y^2\}$ is a $k$-basis of $\uZ(A)$, and in 
particular $\dim_k(\uZ(A))=5$.

\item[{\rm (ii)}] 
The set $\{x,x^2,y,y^2\}$ is a $k$-basis of $J(\uZ(A))$.

\item[{\rm (iii)}] 
The set $\{x^2,y^2\}$ is a $k$-basis of $J(\uZ(A))^2$.

\item[{\rm (iv)}] 
The set $\{y^2\}$ is a $k$-basis of $\soc(\uZ(A))$, and 
$J(\uZ(A))^3=\soc(\uZ(A))$.

\item[\rm (v)] The $k$-algebra
$\uZ(A)$ is a symmetric algebra.
\end{enumerate}
\end{Lemma}

\begin{proof} 
It follows from lemma \ref{symmLemma} that the relative trace map 
$\Tr^A_1$ from $A$ to $Z(A)$ is given by
$$\Tr^A_1(u) = iu\alpha\beta + ju\beta\alpha + \alpha u\beta +
\beta u\alpha + \beta\alpha u j + \gamma u \gamma^2 + 
\gamma^2 u\gamma + \delta u\delta + i u \delta^2$$
for all $u\in$ $A$. One checks, using $\chr(k)=3$,  that 
$$\Tr^A_1(i)=-\Tr^A_1(j)= \beta\alpha - \alpha\beta$$
and that $\Tr^A_1$ vanishes on all basis elements different 
from $i$, $j$. Statement (i) then follows from the relations 
in the quiver of $A$ and Lemma \ref{LemmaA}. The algebra $\uZ(A)$ is 
split local, proving statement (ii), whilst a straightforward 
computation shows both statement (iii) and (iv). Finally, a simple 
verification proves  that the map $s :\uZ(A)\to k$ such that 
$$s(y^2)=1$$ 
and such  that 
$$s(1)=s(x)=s(x^2)=s(y)=0$$ 
is a symmetrising form on  $\uZ(A)$. One verifies also that the dual 
basis with respect to the form $s$ of the basis 
$$\{1,x,y,x^2,y^2\}$$ 
is, in this order, the basis 
$$\{y^2,x^2,y,x,1\}.$$ 
This completes the proof.
\end{proof}

\begin{Remark} \label{reptype}
Note that by a result of Erdmann \cite[I.10.8(i)]{Erdmanntame}, $A$ is of wild representation type.
\end{Remark}

\section{The stable centre of  the group algebra $k(P\rtimes C_2)$.}\label{C2}

Let $k$ be a field of characteristic $3$. Set $P=C_3\times C_3$ and $E$ the
subgroup of $\Aut(P)$ of order $2$ such that the nontrivial element $t$ of
$E$ acts as inversion on $P$. Denote by $H=P\rtimes E$ the corresponding
semidirect product; this is a Frobenius group. 
Denote by $r$ and $s$ a generator of the two factors $C_3$ of $P$.
The following Lemma holds in greater generality (see Remark 4.1 in 
\cite{KessarLinckelmann}); we state only what we need in this paper.

\begin{Lemma} \label{projidealLemma2}
The projective ideal $Z^\pr(kH)$ is one-dimensional, with 
$k$-basis $\{\sum_{x\in P}xt\}$, and we have an isomorphism of 
$k$-algebras 
$$\underline{Z}(kH)\cong (kP)^{E}$$ 
induced 
by the map sending $x+x^{-1}$ in $(kP)^E$ to its image in 
$\underline{Z}(kH)$. In particular, we have $\dim_k(\uZ(kH))=5$, 
and the image of the set $\{1,r+r^2,s+s^2,r^2s+rs^2,rs+r^2s^2\}$ is a 
$k$-basis of $\uZ(kH)$. 
\end{Lemma}

\begin{proof}
The relative trace map $\tn{Tr}_1^{H}$ from $kH$ to $Z(kH)$ satisfies
$\tn{Tr}_1^H=\tn{Tr}_{P}^H\circ\tn{Tr}_1^{P}$. We calculate for all $a \in P$
$$
\tn{Tr}_1^{P}(a)= \sum_{g\in P}gag^{-1}=\sum_{|P|}a=9\cdot a = 0
$$
Thus for every $c \in kP $ we have $\tn{Tr}_1^H(c)=
\tn{Tr}_{P}^H(\tn{Tr}_1^{P}(c))=0$.
On the other hand, for every element of the form $at$ in $H$, where 
$a\in P$, we have

\begin{eqnarray}
\tn{Tr}_1^{H}(at) &=& \sum_{g\in P}g(at)g^{-1} + \sum_{g\in P}(gt)(at)(gt)^{-1} \nonumber\\
&=&  (a + a^{-1})\sum_{x\in P}xt \nonumber\\
&=& 2\cdot\big( \ \sum\limits_{x\in P}xt \ \big)\nonumber
\end{eqnarray}
The conjugacy classes of $G$ are given by $\{1\}$, $\{r,r^2\}$, $\{s,s^2\}$, 
$\{r^2s,rs^2\}$, $\{rs,r^2s^2\}$ and $\{xt \ | \ x\in P\}$. The last statement
follows.
\end{proof}

\begin{Lemma} \label{fixptisomLemma}
There is an isomorphism of $k$-algebras 
$$\uZ(kH)\cong \left(k[x,y]/(x^3,y^3)\right)^{E}$$ 
with inverse induced by the map sending $x$ to $r-1$ and $y$ to 
$s-1$, where the nontrivial element $t$ of $E$ acts by 
$x^t=x^2+2x$ and $y^t=y^2+2y$.
After identifying $x$ and $y$ with their images in $\uZ(kH)$ 
the following statements hold:
\begin{enumerate}
\item[{\rm (i)}] 
The image of the set $\{1, x^2, y^2, xy+x^2y+xy^2, x^2y^2 \}$ is a 
$k$-basis of  $\uZ(kH)$, and in particular we have $\dim_k(\uZ(kH))=5$.

\item[{\rm (ii)}] 
The set $\{x^2, y^2, xy+x^2y+xy^2, x^2y^2\}$ is a $k$-basis of $J(\uZ(kH))$.

\item[{\rm (iii)}] 
The set $\{x^2y^2\}$ is a $k$-basis of $\soc(\uZ(kH))$, and
$J(\uZ(kH))^2=\soc(\uZ(kH))$. In particular, $\dim_k(J(\uZ(kH))^2)=1$.

\item[{\rm (iv)}] 
The $k$-algebra $\uZ(kH)$ is symmetric.
\end{enumerate}
\end{Lemma}

\begin{proof}
By Lemma \ref{projidealLemma2} we have $\uZ(kH)\cong$ $(kP)^E$. 
Since $k$ has characteristic $3$,  we have an isomorphism 
$kP \cong k[x,y]/(x^3,y^3)$ induced by the map given in the 
statement of the lemma. Under this isomorphism, the action of $t$ on $x$ 
and $y$ is given by $x^t=x^2+2x$ and $y^t=y^2+2y$ as stated.
It is straightforward to then verify that this isomorphism gives 
$$r+r^t\mapsto x^2+2,$$ 
$$s+s^t\mapsto y^2+2,$$ 
$$rs+(rs)^t\mapsto 2+x^2+y^2+2xy+2x^2y+2xy^2+x^2y^2,$$ 
$$r^2s+(r^2s)^t \mapsto 2+x^2+y^2+xy+x^2y+xy^2.$$ 
This proves the statement (i) and (ii). 
A straightforward computation proves statement (iii). The final statement 
is given in general in \cite[Corollary 1.3]{KessarLinckelmann}, with an explicit 
symmetrising form $s :\uZ(kH)\to k$ given by  $s(x^2y^2)=1$ and 
sending all other basis elements to $0$.
\end{proof}

\section{Proof of Theorem \ref{main1}}

Theorem \ref{main1} will be an immediate consequence of Theorem 
\ref{main2} and the following result.

\begin{Theorem}\label{main5}
Let $k$ be an algebraically closed field of prime characteristic $p$, and 
let $A$ be the algebra given in Theorem \ref{main2}. Then $A$ is not 
isomorphic to a basic algebra of a block of a finite group algebra over 
$k$.
\end{Theorem}

\begin{proof}
Arguing by contradiction, suppose that $A$ is isomorphic to a basic 
algebra of a block $B$ of $kG$, for some finite group $G$. Denote 
by $P$ a defect group of $B$. By Theorem \ref{main2} we have $p=3$ 
and $P\cong$ $C_3\times C_3$. By Lemma \ref{projcenterLemma}, the
stable centre $\uZ(A)$ is symmetric, hence so is $\uZ(B)$, as $A$ and
$B$ are Morita equivalent.  It follows from 
\cite[Proposition 3.8]{KessarLinckelmann} that we have an algebra 
isomorphism
$$\uZ(A)\cong (kP)^E$$
where $E$ is the inertial quotient of the block $B$. Again by Lemma
\ref{projcenterLemma}, we have $\dim_k((kP)^E)=5$, or equivalently,
$E$ has five orbits in $P$. The list of possible inertial quotients in 
Kiyota's paper \cite{Kiyota} shows that $E$ is isomorphic to one of
$1$, $C_2$, $C_2\times C_2$, $C_4$, $C_8$, $D_8$, $Q_8$, $SD_{16}$.
In all cases except for $E\cong$ $C_2$ is the action of $E$ on $P$ 
determined, up to equivalence,  by the isomorphism class of $E$.
Thus  if $E$ contains a cyclic subgroup of order $4$, then $E$ has 
at most $3$ orbits, and if $E$ is the Klein four group, then $E$ has 
$4$ orbits. Therefore we have $E\cong$ $C_2$. If the nontrivial element 
$t$ of $E$ has a nontrivial fixed point in $P$ (or equivalently, if $t$ 
centralises one of the factors $C_3$ of $P$ and acts as inversion on 
the other),  then $E$ has $6$ orbits. 
Thus $t$ has no nontrivial fixed point in $P$, and the group $H=$
$P\rtimes E$ is  the Frobenius group considered in the previous section.
By a result of Puig \cite[6.8]{Puabelien} (also described in
\cite[Theorem 10.5.1]{LinBlockII}), there is a stable equivalence of 
Morita type between $B$ and $kH$, hence between $A$ and $kH$. 
By a result of Brou\'e \cite[5.4]{BroueEq} (see also
\cite[Corollary 2.17.14]{LinBlockI}), there is an algebra isomorphism
$\uZ(A)\cong$ $\uZ(kH)$. This, however, contradicts the calculations
in the Lemmas \ref{projcenterLemma} and \ref{fixptisomLemma}, which 
show that the dimension of  $J(\uZ(A))^2$ and of $J(\uZ(kH))^2$ are 
different. This contradiction completes the proof.
\end{proof}

\begin{proof}[{Proof of Theorem \ref{main1}}]
Arguing by contradiction, if a defect $P$ of $B$ is not cyclic, then
$P\cong$ $C_3\times C_3$ because the Cartan matrix of $B$
has elementary divisors $9$ and $1$. But then $B$ has a basic
algebra isomorphic to the algebra $A$ in Theorem \ref{main2}.
This, however, is ruled out by Theorem \ref{main5}.
\end{proof}

\section{Further remarks}

Using the arguments of the proof of Theorem \ref{main5} it is possible 
to prove some slightly stronger statements about the
stable equivalence class of the algebra $A$ from Theorem \ref{main2}.

\begin{Proposition} \label{Further1}
Let $k$ be an algebraically closed field of prime characteristic $p$ and let
$A$ be the algebra in Theorem \ref{main2}. Let $P$ be a finite $p$-group,
$E$ a $p'$-subgroup of $\Aut(P)$, and $\tau\in$ $H^2(E;k^\times)$. 
There does not exist a stable equivalence of Morita type between $A$ and
the twisted group algebra $k_\tau(P\rtimes E)$.
\end{Proposition}

\begin{proof}
Arguing by contradiction, suppose that there is a stable equivalence of 
Morita type between $A$ and $k_\tau(P\rtimes E)$. Note that 
$k_\tau(P\rtimes E)$ is a block of a central $p'$-extension of $P\rtimes E$
with defect group $P$, so its Cartan matrix has a determinant divisible
by $|P|$. By \cite[Proposition 4.14.13]{LinBlockI}, the Cartan matrices of
the algebras $A$ and $k_\tau(P\rtimes E)$ have the same determinant,
which is $9$. Since $A$ is clearly not of finite representation type
(cf. Remark \ref{reptype}), it follows that $P$ is not cyclic, hence 
$P\cong$ $C_3\times C_3$. 
Using as before Brou\'e's result  \cite[5.4]{BroueEq}, we have an
isomorphism $\uZ(A)\cong$ $\uZ(k_\tau(P\rtimes E))$. Since
$\uZ(A)$ is symmetric, so is $\uZ(k_\tau(P\rtimes E))$. Since
$k_\tau(P\rtimes E)$ is a block of a central $p'$-extension of 
$P\rtimes E$ with defect group $P$ and inertial quotient $E$, it 
follows again from \cite[Proposition 3.8]{KessarLinckelmann} that
$\uZ(A)\cong$ $(kP)^E$. From this point onward, the rest of the 
proof follows the proof of Theorem \ref{main5},
whence the result.
\end{proof}

\begin{Remark} \label{Further2}
By results of  Rouquier \cite[6.3]{Rouqstable} (see also 
\cite[Theorem A2 ]{Liperm}), for any block $B$ with an elementary 
abelian defect group of rank $2$ there is a stable equivalence of Morita
type between $B$ and its Brauer correspondent, which by a result
of K\"ulshammer \cite{Kuenormal}, is Morita equivalent to a 
twisted semidirect product group algebra as in Proposition 
\ref{Further1}. Thus Theorem \ref{main5} can be obtained as a
consequence of Proposition \ref{Further1} and Rouquier's stable 
equivalence.
\end{Remark}

\begin{Remark} \label{Further3}
A slightly different proof of Theorem \ref{main5} makes use of 
Brou\'e's surjective algebra  homomorphism $Z(B)\to$ $(kZ(P))^E$ 
from  \cite[Proposition III (1.1)]{BroueCoeff}, induced by the Brauer
homomorphism $\Br_P$, where here $P$ is a (not 
necessarily abelian) defect group of a block $B$ of a finite group 
algebra $kG$, with $k$ an algebraically closed field of prime 
characteristic $p$. If $P$ is normal in $G$, then it is easy to see
that Brou\'e's homomorphism is split surjective, but this is not
known in general. If $B$ is a block with  $P$ nontrivial such
that there exists a stable equivalence of Morita type between $B$ 
and its Brauer correspondent, then this implies the existence 
of at least {\it some} split surjective algebra homomorphism 
$\uZ(B)\to$ $kZ(P)^E$. 

Kiyota's list in \cite{Kiyota} shows that if
$A$ were isomorphic to a basic algebra of a block with defect
group $P\cong$ $C_3\times C_3$, then  $E$ would be isomorphic to 
one of $C_2$ or $D_8$ (subcase (b) in Kiyota's list). The case $C_2$ 
can be ruled out as above, and the case $D_8$ 
can be ruled out by using Rouquier's stable equivalence, and by 
showing that if $E\cong$ $D_8$, then $(kP)^E$ is uniserial of 
dimension $3$, but  $\uZ(A)$ admits no split surjective algebra 
homomorphism onto a uniserial algebra of dimension $3$. Note that 
$\uZ(A)$ does though admit a surjective algebra homomorphism onto a 
uniserial algebra of dimension $3$, so the splitting is an essential 
point in this argument, and may warrant further investigation. 
\end{Remark}



\begin{thebibliography}{WWW}



\bibitem{BensonI} D. J. Benson, {\em Representations and cohomology, Vol. I:
Cohomology of groups and modules.} Cambridge studies in advanced
mathematics {\bf 30}, Cambridge University Press (1991). 






\bibitem{BroueCoeff} M. Brou\'e. {\em Brauer coefficients of 
$p$-subgroups associated with a $p$-block of a finite group.}
J. Algebra {\bf 56} (1979), 365--383. 

\bibitem{BroueEq} M. Brou\'e, {\em Equivalences of Blocks of Group
Algebras}, in: {Finite dimensional algebras and related
topics}, Kluwer (1994), 1--26.











\bibitem{Eatonwiki} C. Eaton, 
https://wiki.manchester.ac.uk/blocks/index.php/Blocks\_with\_basic\_algebras\_of\_low\_dimension





\bibitem{Erdmanntame} K. Erdmann, {\em Blocks of tame representation type and related algebras}, Lecture Notes Math. {\bf 1428}, Springer Verlag, Berlin Heidelberg (1990).

















\bibitem{KessarLinckelmann} R. Kessar, M. Linckelmann, {\em On blocks 
with Frobenius inertial quotient}. J. Algebra {\bf 249 (1)}, (2002), 127--146.

\bibitem{Kiyota} M. Kiyota, {\em On $3$-blocks with an elementary 
abelian defect group of order $9$.} J. Fac. Sci. Univ. Tokyo
Sect. IA, Math. {\bf 31} (1984), 33-58.




\bibitem{Kuesmall} B. K\"ulshammer, {\em Symmetric local algebras and
small blocks of finite groups}. J. Algebra {\bf 88} (1984),
190--195.

\bibitem{Kuenormal} B. K\"ulshammer, {\em Crossed products and blocks
with normal defect groups}, Comm. Algebra {\bf 13} (1985), 147--168.






\bibitem{Liperm} M. Linckelmann, {\em Trivial source bimodule rings for
blocks and $p$-permutation equivalences}, Trans. Amer. Math. Soc.
{\bf 361} (2009), 1279--1316.







\bibitem{LinICRA16} M. Linckelmann, {\em Finite-dimensional
algebras arising as blocks of finite group algebras.}
Representations of algebras, Contemp. Math.
{\bf 705}, Amer. Math. Soc., Providence, RI, (2018), 155--188.

\bibitem{LinBlockI} M. Linckelmann, {\em The Block Theory of Finite 
Group Algebras, Volume 1}, LMS Student
Society Texts, \textbf{91}, (2018)

\bibitem{LinBlockII} M. Linckelmann, {\em The Block Theory of Finite 
Group Algebras, Volume 2}, LMS Student
Society Texts, \textbf{92}, (2018)


\bibitem{Okcube} T. Okuyama, {\em On blocks of finite groups with
radical cube zero.} Osaka J. Math. {\bf 23} (1986), 461--465.





\bibitem{Puabelien} L. Puig, {\em Une correspondance de modules
pour les blocs \`a groupes de d\'efaut ab\'eliens}, Geom. Dedicata
{\bf 37} (1991),  no. 1, 9--43.




\bibitem{Rouqstable} R. Rouquier,  {\em Block theory via stable and
Rickard equivalences}, in: Modular representation theory of finite
groups (Charlottesville, VA 1998), (M. J. Collins, B. J. Parshall,
L. L. Scott), DeGruyter, Berlin, 2001, 101-146.






\end{thebibliography}
\end{document}